\documentclass[reqno]{amsart}%
\usepackage{amstext}
\usepackage{amsfonts}
\usepackage{amsmath}
\usepackage{amssymb}
\usepackage{hyperref}
\usepackage{graphicx}%
\numberwithin{equation}{section}
\newtheorem{theorem}{Theorem}[section]

\newtheorem{proposition}[theorem]{Proposition}

\newtheorem{lemma}[theorem]{Lemma}

\begin{document}
\title[Effect of hyperviscosity on turbulence.]{Effect of hyperviscosity on the Navier-Stokes turbulence.}
\author{Abdelhafid Younsi.}
\address{Faculty of Mathematics USTHB, BP32 EL ALIA16111 Algiers, Algeria.}
\email{younsihafid@gmail.com}
\subjclass[2000]{ 76D05, 76F20, 35B30,\ 35B41, 35B65, 37L30, 37K40}
\keywords{Navier-Stokes equations,\thinspace Hyperviscosity, Weak solutions, Attractor
dimension,\ Turbulence models}

\begin{abstract}
In this paper we modified the Navier-Stokes equations by adding a higher order
artificial viscosity term to the conventional system. We first show that the
solution of the regularized system converges strongly to the solution of the
conventional system as the regularization parameter goes to zero, for each
dimension $d\leq4$. Then we show that the use of this artificial viscosity
term leads to truncated the number of degrees of freedom in the long-time
behavior of the solutions to these equations. This result suggests that the
hyperviscous Navier-Stokes system is an interesting model for
three-dimensional fluid turbulence.

\end{abstract}
\maketitle

\section{Introduction}

We regularize the Navier-Stokes equations by adding a higher-order viscosity
term to the conventional system. In this paper we will restrict ourselves to
periodic boundary conditions.%
\begin{equation}%
\begin{array}
[c]{c}%
\dfrac{du_{\varepsilon}}{dt}+\varepsilon\left(  -\triangle\right)
^{l}u_{\varepsilon}-\nu\triangle u_{\varepsilon}+\left(  u_{\varepsilon
}.\nabla\right)  u_{\varepsilon}+\nabla p=f\left(  x\right)  ,\text{ in
}\Omega\times\left(  0,\infty\right) \\
\text{div }u_{\varepsilon}=0,\text{ in }\Omega\times\left(  0,\infty\right)
\text{, }\\
p(x+Le_{i},t)=p(x,t),\text{ }u(x+Le_{i},t)=u(x,t)\text{\ \ }i=1,...,d\text{
\ }t\in\left(  0,\infty\right) \\
u_{\varepsilon}\left(  x,0\right)  =u_{\varepsilon0}\left(  x\right)
,\text{in }\Omega\text{,}%
\end{array}
\label{1}%
\end{equation}
Where $\Omega=\left(  0,L\right)  ^{d}$ and $\left(  e_{1},...,e_{d}\right)  $
is the natural basis of $%
\mathbb{R}
^{d}$. Here $\varepsilon>0$ is the artificial dissipation parameter and
$\nu>0$ is the kinematic viscosity of the fluid, $l>1$. The function
$u_{\varepsilon}\ $is the velocity vector field, $p$ is the pressure, and $f$
is a given force field. For $\varepsilon=0$, the model is reduced to the
Navier-Stokes system.

In Lions \cite{25}, the existence and uniqueness of weak solutions of the
modified Navier-Stokes equations were established for all $l>0$ if $l\geq
\frac{d+2}{4}$, $d$ is the space dimension.

This type of regularization was proposed by Ladyzhenskaya \cite{20} and Lions
\cite{26} who added the artificial hyperviscosity $\left(  -\triangle\right)
^{\frac{l}{2}}$, $l>2$ to the Navier-Stokes system.

Mathematical model for such fluid motion play an important role in theoretical
and computational studies of bipolar fluids \cite{7} and in the regularized
Navier-Stokes equations (see \cite{7, 26, 28} and the references therein).
Hyperviscosity has been widely used for numerical simulations of turbulence
\cite{1,3,5,6} and in computer simulations for oceanic and atmospheric flows
(see \cite{4, 23}) or to control the Navier--Stokes equations \cite{31}.

A well known example of such a result is the viscosity solution method for the
Hamilton-Jacobi equations \cite{26}.

In this paper, we will study the effect of hyperviscosity on the Navier-Stokes
turbulence. First, we show that the solutions of $($\ref{1}$)$ converge
strongly to the corresponding solutions of the Navier--Stokes equations for
$d\leq4$. This result can extend to each domain $\Omega$ with one finite size.

In this result, we show that the conjecture of J.Lions \cite[Remarque 8.2.
SecII]{25} is true, for $d\leq4$. In addition, it is an extension of a result
due to Lions \cite{26} (where only the weak convergence is proved). The
results in this article can be seen as an improved version of the
convergence\ results announced by Yuh-Roung and Sritharan \cite{28,29}, in two
different ways: On the one hand, we consider here a dimension $d\leq4$, on the
other hand the order viscosity term here is $l\geq\sup(\frac{d}{2},\frac
{d+2}{4})$.

Next, we consider the system $($\ref{1}$)$ with $l=2$ i.e. we modified the
$3D$ Navier-Stokes system by adding a fourth order artificial viscosity term
(Laplacian square) and we show the existence of absorbing sets. This fact
implies that the system ($l=2$) possesses a global attractor $\mathfrak{A}%
_{\varepsilon}$.

Finally, we obtain scale-invariant estimates on the Hausdorff and fractal
dimensions of the global attractor $\mathfrak{A}_{\varepsilon}$ independent of
$\varepsilon$ in terms of the Landau--Lifschitz theory \cite{22} of the number
of degrees of freedom in turbulent flow \cite{11, 32}. In fact such an
estimate that improves on the Landau-Lifschitz estimates has already been done
by J. Avrin \cite{1} in which hyperviscous terms are spectrally added to the
Navier-Stokes equations.

Thus we recover the improvement on the cubic power, i.e. get a bound
proportional to $G^{\frac{p}{2}}$ for $p<3$. The latter should be a
possibility, as the attractor results in \cite{1} were not intended to be
optimal in this direction. We would then represent an overlapping result that
is new as far as we know, although readers familiar with the attractor
techniques used may anticipate that such a result is possible in the
hyperviscous case given the existing results in \cite{1} and the expected
improvement in the Sobolev-space estimates in the fixed uniform hyperviscous
case at hand.

In Section 2, we present the relevant mathematical framework for the paper. In
Section 3, we show the convergence of the system $($\ref{1}$)$ to the
conventional Navier--Stokes equations. In Section 4, we consider the
hyperviscous system ($l=2$), we show the existence of a global attractor. In
Section 5, we estimate the dimension of the attractor. Finally, we provide in
Section 6, explicit upper bounds for the dimension of the global attractor of
the modified Navier--Stokes in terms of the relevant physical parameters.

\section{Notations and preliminaries}

In this section we introduce notations and the definitions of standard
functional spaces that will be used throughout the paper. We denote by
$H^{m}\left(  \Omega\right)  $, the Sobolev space of $L$-periodic functions.
These spaces are endowed with the inner product%
\[
\left(  u,v\right)  =%
{\textstyle\sum\limits_{\left\vert \beta\right\vert \leq m}}
(D^{\beta}u,D^{\beta}v)_{L^{2}\left(  \Omega\right)  }%
\]
and the norm%
\[
\left\Vert u\right\Vert _{m}=%
{\textstyle\sum\limits_{\left\vert \beta\right\vert \leq m}}
(\left\Vert D^{\beta}u\right\Vert _{L^{2}\left(  \Omega\right)  }^{2}%
)^{\frac{1}{2}}.
\]
$H^{-m}\left(  \Omega\right)  $ Denote the dual space of $H^{m}\left(
\Omega\right)  $.

We denote by $\dot{H}^{m}\left(  \Omega\right)  $ the subspace of
$H^{m}\left(  \Omega\right)  $ with, zero average%

\[
\dot{H}^{m}\left(  \Omega\right)  =\{u\in H^{m}\left(  \Omega\right)  ;\text{
}%
{\displaystyle\int\limits_{\Omega}}
u\left(  x\right)  dx=0\}.
\]
For $m=0$, we have $\dot{H}^{m}\left(  \Omega\right)  =\dot{L}^{2}\left(
\Omega\right)  $.

\begin{itemize}
\item We introduce the following solenoidal subspaces $V_{s},$ $s\in%
\mathbb{R}
^{+}$ which are important to our analysis
\end{itemize}

the spaces $V_{s}$ as completions of smooth, divergence-free, periodic,
zero-average functions with respect to the $H^{s}$ norms.%
\begin{align*}
V_{0}\left(  \Omega\right)   &  =\{u\in\dot{L}^{2}\left(  \Omega\right)
,\text{div}u=0,u.n\mid_{\Sigma_{i}}=-u.n\mid_{\Sigma_{i+3}},i=1,2,3\};\\
V_{1}\left(  \Omega\right)   &  =\{u\in\dot{H}^{1}\left(  \Omega\right)
,\text{div}u=0,\gamma_{0}u\mid_{\Sigma_{i}}=\gamma_{0}u\mid_{\Sigma_{i+3}%
},i=1,2,3\}.\\
V_{2}\left(  \Omega\right)   &  =\{u\in\dot{H}^{2}\left(  \Omega\right)
,\text{div}u=0,\gamma_{0}u\mid_{\Sigma_{i}}=\gamma_{0}u\mid_{\Sigma_{i+3}%
},\gamma_{1}u\mid_{\Sigma_{i}}=-\gamma_{1}u\mid_{\Sigma_{i+3}},i=1,2,3\},
\end{align*}

see \cite[Chapter III, Section 2]{32}. We refer the reader to R.Temam
\cite{33} for details on these spaces. Here the faces of $\Omega$ are numbered
as%
\[
\Sigma_{i}=\partial\Omega\cap\left\{  x_{i}=0\right\}  \text{ and }%
\Sigma_{i+3}=\partial\Omega\cap\left\{  x_{i}=L\right\}  ,\text{ }i=1,2,3.
\]
Here $\gamma_{0}$, $\gamma_{1}$ are the trace operators and $n$ is the unit
outward normal on $\partial\Omega$.

\begin{itemize}
\item The space $V_{0}$ is endowed with the inner product $\left(  u,v\right)
_{L^{2}\left(  \Omega\right)  }$ and norm $\left\Vert u\right\Vert
=\nolinebreak\left(  u,u\right)  _{L^{2}\left(  \Omega\right)  }^{\frac{1}{2}%
}$.

\item $V_{1}$ Is the Hilbert space with the norm $\left\Vert u\right\Vert
_{1}=\left\Vert u\right\Vert _{V_{1}}$. The norm induced by $\dot{H}%
^{1}\left(  \Omega\right)  $\ and the norm $\left\Vert \nabla u\right\Vert
$\ are equivalent in $V_{1}$.

\item $V_{2}$ Is the Hilbert space with the norm $\left\Vert u\right\Vert
_{2}=\left\Vert u\right\Vert _{V_{2}}$. In $V_{2}$ the norm induced by
$\dot{H}^{2}\left(  \Omega\right)  $ is equivalent to the norm $\left\Vert
\triangle u\right\Vert $.
\end{itemize}

$V_{s}^{\prime}$ Denote the dual space of $V_{s}$. We denote by $A$ the Stokes
operator%
\[
Au=-\triangle u\text{ for }u\in D\left(  A\right)  .
\]
We recall that the operator $A$ is a closed positive self-adjoint unbounded
operator, with $D\left(  A\right)  =\left\{  u\in V_{0}\text{, }Au\in
V_{0}\right\}  $. We have in fact,
\[
D\left(  A\right)  =\dot{H}^{2}\left(  \Omega\right)  \cap V_{0}=V_{2}\text{.}%
\]
The eigenvalues of $A$ are $\left\{  \lambda_{j}\right\}  _{j=1}^{j=\infty}$,
$0$ $<$ $\lambda_{1}\leq\lambda_{2}\leq...$and the corresponding orthonormal
set of eigenfunctions $\left\{  w_{j}\right\}  _{j=1}^{j=\infty}$ is complete
in $V_{0}$%
\[
Aw_{j}=\lambda_{j}w_{j},\ \ w_{j}\in D(A_{1}),\forall j.
\]
The spectral theory of $A$ allows us to define the powers $A^{l}$ of $A$ for
$l\geq1$,\newline$A^{l}$ is an unbounded self-adjoint operator in $V_{0}$ with
a domain $D(A^{l})$ dense in $V_{2}\subset V_{0}$. We set here%
\[
A^{l}u=\left(  -\triangle\right)  ^{l}u\text{\ for }u\in D\left(
A^{l}\right)  =V_{2l}\cap V_{0}\text{.}%
\]
The space $D\left(  A^{l}\right)  $ is endowed with the scalar product and the
norm%
\[
\left(  u,v\right)  _{D(A^{l})}=\left(  A^{l}u,A^{l}v\right)  \text{,
}\left\Vert u\right\Vert _{D\left(  A^{l}\right)  }=\{\left(  u,v\right)
_{D\left(  A^{l}\right)  }\}^{\frac{1}{2}}.
\]
Let us now define the trilinear form $b(.,.,.)$ associated with the inertia
terms%
\[
b\left(  u,v,w\right)  =\sum_{i,j=1}^{3}%
{\displaystyle\int\limits_{\Omega}}
u_{i}\frac{\partial v_{j}}{\partial x_{_{i}}}w_{j}dx.
\]
The continuity property of the trilinear form enables us to define (using
Riesz representation Theorem) a bilinear continuous operator $B\left(
u,v\right)  $; $V_{2}\times V_{2}\rightarrow V_{2}^{\prime}$ will be defined
by
\begin{equation}
\left\langle B\left(  u,v\right)  ,w\right\rangle =b\left(  u,v,w\right)
,\text{ }\forall w\in V_{2}\text{.} \label{2}%
\end{equation}
Recall that for $u$ satisfying $\nabla.u=0$ we have%
\begin{equation}
b\left(  u,u,u\right)  =0\text{ and }b\left(  u,v,w\right)  =-b\left(
u,w,v\right)  \text{.} \label{3}%
\end{equation}
Hereafter, $c_{i}$ for $i\in%
\mathbb{N}
$, will denote a dimensionless scale invariant positive constant which might
depend on the shape of the domain. Similarly, the trilinear form $b\left(
u,v,w\right)  $ satisfies the well-known inequalities (see, for instance,
\cite[Lemma 61.1]{30} and \cite{8, 33})%
\begin{equation}
\left\vert b(u,v,u)\right\vert \leq c_{1}\left\Vert u\right\Vert ^{\frac{1}%
{2}}\left\Vert u\right\Vert _{1}^{\frac{3}{2}}\left\Vert v\right\Vert
_{1}\text{ for all }u,v\in V. \label{4}%
\end{equation}

The trilinear form $b\left(  .,.,.\right)  $ is continuous on $\dot{H}^{m_{1}%
}\left(  \Omega\right)  \times\dot{H}^{m_{2}+1}\left(  \Omega\right)
\times\dot{H}^{m_{3}}\left(  \Omega\right)  $, $m_{i}\geq0$
\begin{equation}
\left\vert b\left(  u,v,w\right)  \right\vert \leq c_{2}\left\Vert
u\right\Vert _{m_{1}}\left\Vert v\right\Vert _{m_{2}+1}\left\Vert w\right\Vert
_{m_{3}}\text{ \ },\text{ }m_{3}+m_{2}+m_{1}\geq\frac{3}{2} \label{5}%
\end{equation}
see \cite{21}. We recall some well known inequalities that we will be using in
what follows.

Agmon inequality (see, e.g., \cite{8})%
\begin{equation}
\left\Vert u\right\Vert _{\infty}\leq c_{3}\left\Vert u\right\Vert _{1}%
^{\frac{1}{2}}\left\Vert Au\right\Vert ^{\frac{1}{2}}\text{ for all }u\in
V_{2}. \label{6}%
\end{equation}
Young's inequality%
\begin{equation}
ab\leq\tfrac{\sigma}{p}a^{p}+\tfrac{1}{q\sigma^{\frac{q}{p}}}b^{q}%
,a,b,\sigma>0,\text{ }p>1,\text{ }q=\tfrac{p}{p-1}. \label{7}%
\end{equation}
Poincar\'{e} inequality%
\begin{equation}
\lambda_{1}\left\Vert u\right\Vert ^{2}\leq\Vert A^{\frac{1}{2}}u\Vert
^{2}\text{\ for all }u\in V\text{.} \label{8}%
\end{equation}
To prove uniform bounds on different norms we use the uniform Gronwall Lemma
for proof see \cite[Lemma III 1.1]{32}.

\begin{lemma}
(The Uniform Gronwall Lemma) Let $g$, $h$, $y$ be three positive locally
integrable functions on $\left(  t_{0,}+\infty\right)  $ which satisfy%
\[
\frac{dy}{dt}\leq gy+h\text{ for }t\geq t_{0}\text{ and }%
\]%
\[
\int_{t}^{t+r}g\left(  s\right)  ds\leq a_{1}\text{, }\int_{t}^{t+r}h\left(
s\right)  ds\leq a_{2}\text{, }\int_{t}^{t+r}y\left(  s\right)  ds\leq
a_{3}\text{\ for }t\geq t_{0},
\]
where $a_{1}$, $a_{2}$, $a_{3}$ are positive constants. Then%
\[
y\left(  t+r\right)  \leq\left(  \frac{a_{3}}{r}+a_{2}\right)  \exp\left(
a_{1}\right)  \text{ for }t\geq t_{0.}%
\]

\end{lemma}

Denoting by $G$ the dimensionless Grashoff number \cite{10}, $G=\dfrac
{\left\Vert f\right\Vert }{\nu^{2}\lambda_{1}^{\frac{3}{4}}}$ in 3D, (see e.g.
\cite{1,11, 33}). Tthis number measures the relative strength of the forcing
and viscosity.

\section{Strong convergence for the hyperviscous system}

In this Section, we give a new Theorem which insures the strong convergence of
the solutions of the system $($\ref{1}$)$ to the corresponding solutions of
the Navier--Stokes equations for $d\leq4$. This result can extend to each
domain $\Omega$ with one finite size. Moreover, we show that $u_{\varepsilon
}\in C\left(  0,T;V_{0}\right)  $.

Using the operators defined above, we can write the modified system $($%
\ref{1}$)$ in the evolution form%

\begin{align}
\dfrac{du_{\varepsilon}}{dt}+\varepsilon A^{l}u_{\varepsilon}+\nu
Au_{\varepsilon}+B\left(  u_{\varepsilon},u_{\varepsilon}\right)   &
=f\left(  x\right)  \text{, in }\Omega\times\left(  0,\infty\right)
\label{9}\\
u_{\varepsilon0}\left(  x\right)   &  =u_{\varepsilon0}\text{ \ ,\ in }%
\Omega\text{.} \label{10}%
\end{align}

The existence and uniqueness results for initial value problem $($\ref{1}$)$
can be found in \cite{25}, \cite[Chap.1, Remarque 6.11]{26}.\newline The
following theorem collects the main result in this work

\begin{theorem}
\label{Theorem1} For $l\geq\frac{d+2}{4}$, $d$ is the space dimension, for
$\varepsilon>0$ fixed, $f\in\nolinebreak L^{2}\left(  0,T;V_{0}^{\prime
}\right)  $ and $u_{\varepsilon0}\in V_{0}$ be given. There exists a unique
weak solution of $($\ref{1}$)$ which satisfies%
\[
u_{\varepsilon}\in L^{2}\left(  0,T;V_{l}\right)  \cap L^{\infty}\left(
0,T;V_{0}\right)  ,\forall T>0.
\]

\end{theorem}

Notice that the conventional Navier-Stokes system can be written in the
evolution form%

\begin{align}
\dfrac{du}{dt}+\nu Au+B\left(  u,u\right)   &  =f\left(  x\right)  \text{, in
}\Omega\times\left(  0,\infty\right) \label{11}\\
u\left(  0\right)   &  =u_{0}\text{ \ ,\ in }\Omega\text{.} \label{12}%
\end{align}

\begin{theorem}
For $d\leq4$, for $f\in L^{2}(0;T;V_{0})$ and $u_{0}\in V_{0}$ be given. There
exists a weak solution of $($\ref{11}$)$-$($\ref{12}$)$ which satisfies $u\in
L^{\infty}(0;T;V_{0})\cap L^{2}(0;T;V_{1})$, for $T>0$. For $d=2$, $u$ is
unique (J. Lions \cite{25}).
\end{theorem}

We will establish various estimates uniform in $\varepsilon$ for the solutions
of the modified Navier Stokes. These bounds will be used to establish the
limit of these solutions to the conventional Navier Stokes equations.

\begin{proposition}
\label{Proposition1}For $d\leq4$ and for $\varepsilon>0$ fixed, $f\in
L^{2}\left(  0,T;V_{0}\right)  $ and $u_{\varepsilon0}\in V_{0}$. The weak
solution $u_{\varepsilon}\left(  t\right)  $ of the modified Navier-Stokes
equations satisfy\newline$i)$ $u_{\varepsilon}$ \textit{is uniformly bounded
in} $L^{\infty}\left(  0,T;V_{0}\right)  $,\newline$ii)$ $u_{\varepsilon}$
\textit{is uniformly bounded in} $L^{2}\left(  0,T;V_{1}\right)  $.
\end{proposition}

We need the following Lemma proved in R. Temam \cite[Lemma 4.1.ChIII,Sec4]{33}.

\begin{lemma}
\label{Lemma1}The form $b$ is trilinear continuous on $V\times V\times V_{s}$
if $s\geq\frac{d}{2}$ and
\end{lemma}

\[
\Vert b(u,v,w)\Vert\leq c_{4}\Vert u\Vert\Vert v\Vert_{1}\Vert w\Vert_{s}.
\]

Applying\ Lemma
\ref{Lemma1}
we obtain

\begin{lemma}
\label{Lemma2}Let $u_{\varepsilon}\left(  t\right)  $ be a weak solution of
the modified Navier-Stokes system. Then $B\left(  u_{\varepsilon}\right)  $
belongs to $L^{2}\left(  0,T;V_{l}^{\prime}\right)  $ for $l\geq\frac{d}{2}$.
\end{lemma}

\bigskip

\begin{proof}
By the definition of the operator $B$ and the above Lemma, we get%
\begin{align*}
\left\vert \left\langle B(u\left(  t\right)  ,v)\right\rangle \right\vert  &
=\left\vert b(u\left(  t\right)  ,u\left(  t\right)  ,v)\right\vert \\
&  \leq c_{4}\Vert u\left(  t\right)  \Vert\Vert u\left(  t\right)  \Vert
_{1}\Vert v\Vert_{V_{l}^{\prime}},\text{ }\forall v\in V_{l}.
\end{align*}
We set
\[
B(u\left(  t\right)  )=B(u\left(  t\right)  ,u\left(  t\right)  ),
\]
thus%
\[
\Vert B(u\left(  t\right)  )\Vert_{V_{l}^{\prime}}\leq c_{4}\Vert u\left(
t\right)  \Vert\Vert u\left(  t\right)  \Vert_{1}\text{ for }0\leq t\leq T.
\]

\end{proof}

\begin{lemma}
\label{Lemma3}If $f\in L^{2}\left(  0,T;V_{1}^{\prime}\right)  $, then, for
any solution $u_{\varepsilon}\left(  t\right)  $ of problem $($\ref{1}$)$ the
time derivative $\dfrac{du_{\varepsilon}}{dt}$ \textit{is uniformly bounded
in} $L^{2}\left(  0,T;V_{l}^{\prime}\right)  $.
\end{lemma}

\begin{proof}
Due to Lemma \ref{Lemma2} $B\left(  u_{\varepsilon}\right)  $ belongs to
$L^{2}\left(  0,T;V_{l}^{\prime}\right)  $, since $f-\varepsilon
A^{l}u_{\varepsilon}-\nu Au_{\varepsilon}\ $belongs to $L^{2}\left(
0,T;V_{l}^{\prime}\right)  $, this implies that $\dfrac{du_{\varepsilon}}%
{dt}\ $belongs to $L^{2}\left(  0,T;V_{l}^{\prime}\right)  $.
\end{proof}

\begin{lemma}
$u_{\varepsilon}$ is almost everywhere equal to a continuous function from
$[0,T]$ to the space\ $V_{0}$.
\end{lemma}

\begin{proof}
Since $u_{\varepsilon}\in L^{2}\left(  0,T;V_{1}\right)  \cap L^{\infty
}\left(  0,T;V_{0}\right)  $ and $\dfrac{du_{\varepsilon}}{dt}\in L^{2}\left(
0,T;V_{l}^{\prime}\right)  $, the weak continuity in $V_{0}$ is a direct
consequence of \cite[Lemma 1.4.ChIII,Sec1]{33}.

Similarly, it follows that $u_{\varepsilon}\left(  0\right)  $ converges to
$u\left(  0\right)  $ in $V_{0}$, and since $u_{\varepsilon0}$ converges to
$u_{0}$ in $V_{l}^{\prime}$, we conclude that $u(0)=u_{0}.$
\end{proof}

Now we prove the strong convergence. It follows from $ii)$ of Proposition
\ref{Proposition1} and from Lemma \ref{Lemma3}, that easily \newline%
$u_{\varepsilon_{n}}\in\mathcal{X=}\{u_{\varepsilon_{n}}\in L^{2}(0,T;V_{1}),$
$\dfrac{du_{\varepsilon_{n}}}{dt}\in L^{2}\left(  0,T;V_{l}^{\prime}\right)
\}$ with bounds independent of $\varepsilon_{n}$. Hence\newline$\left(
i\right)  $\ $\ u_{\varepsilon_{n}}\rightarrow u$ in $L^{2}(0,T;V_{l})$
\ weakly; and $\left(  ii\right)  $ $\dfrac{du_{\varepsilon_{n}}}%
{dt}\rightarrow\dfrac{du}{dt}$ in $L^{2}(0,T;V_{l}^{\prime})$
\ weakly;\newline These two properties allow us to establish the strong
convergence result.

The proof of the following theorem can be found in R. Temam \cite[Theorem 2.1,
Chapter III, Sec\ 2]{33}.

\begin{theorem}
\label{Theorem3}The injection of $\mathcal{X}=\{u\in L^{2}(0,T;V_{1}),$
$\dfrac{du_{\varepsilon}}{dt}\in L^{2}\left(  0,T;V_{l}^{\prime}\right)  \}$
into $\mathcal{Y}=\left\{  u\in L^{2}(0,T;V_{0})\right\}  $ is compact.
\end{theorem}

By virtue of the above estimates and the compactness Theorem \ref{Theorem3}.
We can now state our first result.

\begin{theorem}
\label{Theorem4} For $l\geq\sup(\frac{d}{2},\frac{d+2}{4})$ and for $d\leq4$,
the weak solution $u_{\varepsilon}$ of the modified Navier-Stokes equations
$($\ref{1}$)$ given by Theorem \ref{Theorem1} converges strongly in
$L^{2}(0,T;V_{0})$ as $\varepsilon\rightarrow0$ to $u$ a weak solution of the
system $($\ref{11}$)$-$($\ref{12}$)$.
\end{theorem}

\begin{proof}
Theorem \ref{Theorem1} and Lemma \ref{Lemma1} are satisfied for $l\geq
\sup(\frac{d}{2},\frac{d+2}{4})$. We use part $ii)$ of Proposition
\ref{Proposition1} and Lemma \ref{Lemma3} we can deduce that the weak
solutions\newline$u_{\varepsilon_{n}}\in\mathcal{X=}\{u_{\varepsilon_{n}}\in
L^{2}(0,T;V_{1}),$ $\dfrac{du_{\varepsilon_{n}}}{dt}\in L^{2}\left(
0,T;V_{l}^{\prime}\right)  \}$. Hence, the compactness Theorem \ref{Theorem3}
implies the strong convergence in $L^{2}(0,T;V_{0})$.
\end{proof}

The following proposition is a consequence of Proposition \ref{Proposition1}.

\begin{proposition}
\label{Proposition2} $\forall w\in L^{2}\left(  0,T;V_{1}\right)  $,
$\forall\dfrac{dw}{dt}\in L^{2}\left(  0,T;V_{1}^{\prime}\right)  $%
\newline$a)\ \lim\limits_{n\rightarrow\infty}\int_{0}^{T}(\dfrac
{du_{\varepsilon_{n}}\left(  t\right)  }{dt},w)dt=\int_{0}^{T}(\dfrac
{du\left(  t\right)  }{dt},w\left(  t\right)  )dt,$\newline$b)\ \lim
\limits_{n\rightarrow\infty}\int_{0}^{T}\left(  \nabla u_{\varepsilon_{n}%
}\left(  t\right)  ,\nabla w\left(  t\right)  \right)  dt=\int_{0}^{T}\left(
\nabla u\left(  t\right)  ,\nabla w\left(  t\right)  \right)  dt,$%
\newline$c)\ \lim\limits_{n\rightarrow\infty}\int_{0}^{T}b\left(
u_{\varepsilon_{n}}\left(  t\right)  ,u_{\varepsilon_{n}}\left(  t\right)
,w\left(  t\right)  \right)  dt=\int_{0}^{T}b\left(  u\left(  t\right)
,u\left(  t\right)  ,w\left(  t\right)  \right)  dt$.
\end{proposition}

Let us now establish the limit of the equations $($\ref{9}$)$ as
$\varepsilon_{n}\rightarrow0$. Taking the inner product of $($\ref{9}$)$ with
a test function $\varphi\in\mathcal{D(}0,T;\mathcal{D}(A^{\frac{l}{2}}%
))$\thinspace\ then integrate by parts and using the convergence Proposition
\ref{Proposition2} we can pass to the limit as $\varepsilon_{n}\rightarrow0,$
we get $-\int_{0}^{T}\left(  u,\varphi^{\prime}\right)  dt+\nu\int_{0}%
^{T}\left(  \nabla u,\nabla\varphi\right)  dt+\int_{0}^{T}b\left(
u,u,\varphi\right)  dt=\int_{0}^{T}\left\langle f,\varphi\right\rangle dt.$

Here the term $\varepsilon_{n}\int_{0}^{T}(A^{\frac{l}{2}}u_{\varepsilon_{n}%
}\left(  t\right)  ,A^{\frac{l}{2}}\varphi\left(  t\right)  )dt$ goes to $0$
as $\varepsilon_{n}\rightarrow0$.

Since the weak solution $u_{\varepsilon_{n}}\in L^{2}\left(  0,T;V_{1}\right)
$ with bound uniform in $\varepsilon_{n}$ and we get $\varepsilon_{n}\int
_{0}^{T}\left\vert (A^{\frac{l}{2}}u_{\varepsilon_{n}},A^{\frac{l}{2}}%
\varphi)\right\vert dt\leq\varepsilon_{n}\int_{0}^{T}\left\vert \left(
u_{\varepsilon_{n}},A^{l}\varphi\right)  \right\vert dt\leq c\varepsilon_{n}.$

Since $u\in L^{2}\left(  0,T;V_{1}\right)  \cap L^{\infty}\left(
0,T;V_{0}\right)  $, we can conclude that $u$ is indeed a weak solution for
the conventional Navier-Stokes equations.

\section{The hyperviscous Navier-Stokes system and attractors}

Now, we consider modifications of the 3D Navier-Stokes system by adding a
fourth order artificial viscosity term (Laplacian square) depending on a small
parameter $\varepsilon$ to the conventional system.%

\begin{equation}%
\begin{array}
[c]{c}%
\dfrac{du_{\varepsilon}}{dt}+\varepsilon A^{2}u_{\varepsilon}+\nu
Au_{\varepsilon}+B\left(  u_{\varepsilon},u_{\varepsilon}\right)  =f\left(
x\right)  ,\text{ in }\Omega\times\left(  0,\infty\right) \\
\text{div }u_{\varepsilon}=0,\text{ in }\Omega\times\left(  0,\infty\right)
\text{, }u_{\varepsilon}\left(  x,0\right)  =u_{\varepsilon0}\left(  x\right)
,\text{ in }\Omega\text{,}\\
p(x+Le_{i},t)=p(x,t),\text{ }u(x+Le_{i},t)=u(x,t)\text{\ \ }i=1,2,3.\ t\in
\left(  0,\infty\right)
\end{array}
\label{13}%
\end{equation}
where $\Omega=\left(  0,L\right)  ^{3}$. In this section we will show
the\ existence of the compact global attractor $\mathfrak{A}_{\varepsilon}$
associated with the semigroup $S_{\varepsilon}\left(  t\right)  $ generated by
the problem $($\ref{13}$)$. (For the theory of global attractors see \cite{2},
\cite{8}, \cite{14}, \cite{18}, \cite{27}, \cite{30}, \cite{32}.).

For $\varepsilon=0$ weak solutions of problem are known to exist by a basic
result by J. Leray from 1934 \cite{24}, only the uniqueness of weak solutions
remains as an open problem. Then the known theory of global attractors of
infinite dimensional dynamical systems is not applicable to the 3D
Navier--Stokes system.

The theory of trajectory attractors for evolution partial differential
equations was developed in \cite{30}, which the uniqueness theorem of
solutions of the corresponding initial-value problem is not proved yet, e.g.
for the 3D Navier--Stokes system (see, for instance,\cite{14,30}). Such
trajectory attractor is a classical global attractor but in the space of weak solutions.

The problem of upper semicontinuity of global attractors for the 2D with
periodic boundary conditions was discussed by Yuh-Roung Ou and S. S.
Sritharan\ in \cite{28}. For related results which use the theory has been
introduced by Foias, Sell, and Temam in \cite{12,32} to show that the system
$($\ref{1}$)$ possesses an inertial manifold (see \cite{1,29,32}).

The existence and uniqueness results for initial value problem $($\ref{13}$)$
are consequence of Theorem \ref{Theorem4} for $l=2$ and $d=3$.

\begin{theorem}
\label{Theorem5}Let $\Omega\subset%
\mathbb{R}
^{3}$, and let $f\in L^{2}\left(  0,T;V_{2}^{\prime}\right)  $ and
$u_{\varepsilon0}\in V_{0}$ be given. Then there exists a unique weak solution
of $($\ref{13}$)$\ which satisfies\newline$u_{\varepsilon}\in C\left(  \left[
0,T\right]  ;V_{0}\right)  \cap L^{2}\left(  0,T;V_{2}\right)  ,\forall T>0$.
Then as $\varepsilon\rightarrow0$, the solution $u_{\varepsilon}$ converges to
a weak solution of the Navier-Stokes equations.
\end{theorem}

Now, we show that the semigroup $S_{\varepsilon}\left(  t\right)  $ has an
absorbing ball in $V_{0}$ and an absorbing ball in $V_{1}$. Then we show that
$S_{\varepsilon}\left(  t\right)  $ admits a compact attractor in $V_{0}$ for
each $\varepsilon\geq0$.

We take the inner product of $($\ref{13}$)$ with $u_{\varepsilon}$, we obtain
the energy equality%
\[
\frac{d}{dt}\left\Vert u_{\varepsilon}\right\Vert ^{2}+2\varepsilon\left\Vert
Au_{\varepsilon}\right\Vert ^{2}+2\nu\left\Vert \nabla u_{\varepsilon
}\right\Vert ^{2}=2\left(  f,u_{\varepsilon}\right)  .
\]
Here we have used the fact that $b\left(  u_{\varepsilon},u_{\varepsilon
},u_{\varepsilon}\right)  =0$. By applying Young's inequality and the
Poincar\'{e} Lemma, we get%
\begin{equation}
\frac{d}{dt}\left\Vert u_{\varepsilon}\right\Vert ^{2}+2\varepsilon\left\Vert
Au_{\varepsilon}\right\Vert ^{2}+\nu\left\Vert \nabla u_{\varepsilon
}\right\Vert ^{2}\leq\frac{\left\Vert f\right\Vert ^{2}}{\nu\lambda_{1}},
\label{14}%
\end{equation}
we drop the term $2\varepsilon\left\Vert Au_{\varepsilon}\right\Vert ^{2}$, we
obtain%
\[
\frac{d}{dt}\left\Vert u_{\varepsilon}\right\Vert ^{2}+\nu\lambda
_{1}\left\Vert u_{\varepsilon}\right\Vert ^{2}\leq\frac{\left\Vert
f\right\Vert ^{2}}{\nu\lambda_{1}},
\]
by integrating the above inequality from $0$ to $t$,we get%
\begin{equation}
\left\Vert u_{\varepsilon}\left(  t\right)  \right\Vert ^{2}\leq\left\Vert
u_{\varepsilon0}\right\Vert ^{2}e^{-\nu\lambda_{1}t}+\rho_{0}^{2}\left(
1-e^{-\nu\lambda_{1}t}\right)  ,\text{ }t>0, \label{15}%
\end{equation}
where $\rho_{0}=\dfrac{1}{\nu\lambda_{1}}\left\Vert f\right\Vert $. Hence for
any ball $B_{R_{0}}=\left\{  u_{\varepsilon0}\in V_{0};\text{ }\left\Vert
u_{\varepsilon0}\right\Vert \leq R_{0}\right\}  $ there is a ball $B\left(
0,\delta_{0}\right)  $ in $V_{0}$ centered at origin with radius $\delta
_{0}>\rho_{0}$ $\left(  R_{0}>\delta_{0}\right)  $ such that
\begin{equation}
S_{\varepsilon}(t)B_{R_{0}}\subset B_{r_{0}}\text{ for }t\geq t_{0}\left(
B_{R_{0}}\right)  =\frac{1}{\nu\lambda_{1}}\log\frac{R_{0}^{2}-\rho_{0}^{2}%
}{\delta_{0}^{2}-\rho_{0}^{2}}. \label{16}%
\end{equation}
The ball $B_{\delta_{0}}$ is said to be absorbing and invariant under the
action of $S_{\varepsilon}(t)$.

Taking the limit in $($\ref{15}$)$ we get,%
\begin{equation}
\lim\sup_{t\rightarrow\infty}\left\Vert u_{\varepsilon}\left(  t\right)
\right\Vert \leq\rho_{0}\text{.} \label{17}%
\end{equation}
We integrate $($\ref{14}$)$ from $t$ to $t+r$, we obtain for $u_{\varepsilon
0}\in B_{R_{0}}$
\begin{equation}
\int_{t}^{t+r}\left\Vert u_{\varepsilon}\right\Vert _{1}^{2}ds\leq\frac{1}%
{\nu}(\frac{r\left\Vert f\right\Vert ^{2}}{\nu\lambda_{1}}+\left\Vert
u_{\varepsilon}\left(  t\right)  \right\Vert ^{2})\text{, }\forall r>0,\text{
}\forall t\geq t_{0}(B_{R_{0}}). \label{18}%
\end{equation}
With the use of $($\ref{17}$)$ we conclude that
\begin{equation}
\lim\sup_{t\rightarrow\infty}\int_{t}^{t+r}\left\Vert u_{\varepsilon
}\right\Vert _{1}^{2}ds\leq\frac{r}{\nu^{2}\lambda_{1}}\left\Vert f\right\Vert
^{2}+\frac{\left\Vert f\right\Vert ^{2}}{\nu^{3}\lambda_{1}^{2}}, \label{19}%
\end{equation}
from which we obtain
\begin{equation}
\lim\sup_{t\rightarrow\infty}\frac{1}{t}\int_{0}^{t}\left\Vert u_{\varepsilon
}\right\Vert _{1}^{2}ds\leq\frac{\left\Vert f\right\Vert ^{2}}{\nu^{2}%
\lambda_{1}}, \label{20}%
\end{equation}
this verifies that the left-hand side is finite.

To show that the semigroup $S_{\varepsilon}(t)$ has an absorbing set in
$V_{1}$, we consider the strong solutions and take the inner product of
$($\ref{13}$)$ with\ $Au_{\varepsilon}$, we obtain%
\begin{equation}
\frac{1}{2}\frac{d}{dt}\Vert A^{\frac{1}{2}}u_{\varepsilon}\Vert
^{2}+\varepsilon\Vert A^{\frac{3}{2}}u_{\varepsilon}\Vert^{2}+\nu\Vert
Au_{\varepsilon}\Vert^{2}=-b(u_{\varepsilon},u_{\varepsilon},Au_{\varepsilon
})+(f,Au_{\varepsilon}). \label{21}%
\end{equation}
By applying Young's inequality, we get%
\begin{align*}
(f,Au_{\varepsilon})  &  \leq\left\Vert f\right\Vert \left\Vert
Au_{\varepsilon}\right\Vert \\
&  \leq\frac{\nu}{4}\left\Vert Au_{\varepsilon}\right\Vert ^{2}+\frac{1}{\nu
}\left\Vert f\right\Vert ^{2}.
\end{align*}
By using the Agmon's inequality $($\ref{6}$)$ and Young's inequality we can
estimate the last term in the left-hand side of $($\ref{21}$)$ as follows%
\begin{align*}
\left\vert b(u_{\varepsilon},u_{\varepsilon},Au_{\varepsilon})\right\vert  &
\leq\left\Vert u_{\varepsilon}\right\Vert _{\infty}\left\Vert u_{\varepsilon
}\right\Vert _{1}\left\Vert Au_{\varepsilon}\right\Vert \\
&  \leq c_{4}\left\Vert u_{\varepsilon}\right\Vert _{1}^{\frac{3}{2}%
}\left\Vert Au_{\varepsilon}\right\Vert ^{\frac{3}{2}}\\
&  \leq\frac{\nu}{4}\left\Vert Au_{\varepsilon}\right\Vert ^{2}+c_{4}%
\left\Vert u_{\varepsilon}\right\Vert _{1}^{6}.
\end{align*}
Hence we obtain from $($\ref{21}$)$
\[
\frac{d}{dt}\left\Vert u_{\varepsilon}\right\Vert _{1}^{2}+2\varepsilon\Vert
A^{\frac{3}{2}}u_{\varepsilon}\Vert^{2}+\nu\left\Vert Au_{\varepsilon
}\right\Vert ^{2}\leq\frac{2}{\nu}\left\Vert f\right\Vert ^{2}+2c_{5}%
\left\Vert u_{\varepsilon}\right\Vert _{1}^{6}\text{.}%
\]
Dropping the positive terms associated with $\varepsilon$ we have
\begin{equation}
\frac{d}{dt}\left\Vert u_{\varepsilon}\right\Vert _{1}^{2}+\nu\left\Vert
A_{1}u_{\varepsilon}\right\Vert ^{2}\leq\frac{2\left\Vert f\right\Vert ^{2}%
}{\nu}+2c_{4}\left\Vert u_{\varepsilon}\right\Vert _{1}^{6} \label{22}%
\end{equation}
we apply the uniform Gronwall Lemma to $($\ref{22}$)$\ with%
\[
g=2c_{4}\left\Vert u_{\varepsilon}\right\Vert _{1}^{4},\text{ }h=\frac
{2\left\Vert f\right\Vert ^{2}}{\nu}\text{, }y=\left\Vert u_{\varepsilon
}\right\Vert _{1}^{2}.
\]

For $n=3$, $m=2$ and $\theta$ $=\frac{1}{2}$, in \cite[Formula (6.167)]{26},
we get $q_{\theta}=6$ wich means $u_{\varepsilon}\in L^{6}\left(
0,T;V_{1}\right)  $ then $u_{\varepsilon}\in L^{4}\left(  0,T;V_{1}\right)  ,$
thus
\[
a_{4}=\left\Vert u\right\Vert _{L^{4}\left(  0,T;V_{1}\right)  }.
\]
Thanks to $($\ref{15}$)$-$($\ref{19}$)$ we estimate the quantities $a_{1}$,
$a_{2}$, $a_{3}$ in Gronwall Lemma by%
\[
a_{1}=2c_{4}a_{4},\text{ }a_{2}=\frac{2r\left\Vert f\right\Vert ^{2}}{\nu
},\text{ }a_{3}=\frac{r\left\Vert f\right\Vert ^{2}}{\nu^{2}\lambda_{1}}%
+\frac{\left\Vert f\right\Vert ^{2}}{\nu^{3}\lambda_{1}^{2}}.
\]
Then we obtain
\[
\left\Vert u_{\varepsilon}\left(  t\right)  \right\Vert _{1}^{2}\leq
(\frac{a_{3}}{r}+a_{2})\exp\left(  a_{1}\right)  =R_{1}^{2}\text{ for }t\geq
t_{0},\text{ }t_{0}\text{ as in }(\text{\ref{16}})\text{.}%
\]
Hence, for any ball $B_{R_{1}}$, there exists a ball $B_{\delta_{1}}$, in
$V_{1}$ centered at origin with radius $R_{1}>\delta_{1}>\rho_{1}$ such that%
\[
S_{\varepsilon}(t)B_{R_{1}}\subset B_{\delta_{1}}\text{ for }t\geq
t_{1}\left(  B_{R_{0}}\right)  =t_{0}\left(  B_{R_{0}}\right)  +1+\frac{1}%
{\nu\lambda_{1}}\log\frac{R_{1}^{2}-\rho_{1}^{2}}{\delta_{1}^{2}-\rho_{1}^{2}%
}.
\]
The ball $B_{\delta_{1}}$ is said to be absorbing and invariant for the
semigroup $S_{\varepsilon}(t)$.

Furthermore, if $B$ is any bounded set of $V_{0}$, then $S_{\varepsilon
}(t)B\subset B_{\delta_{1}}$ for $t\geq t_{1}\left(  B,R_{0}\right)  $, this
shows the existence of an absorbing set in $V_{1}$. Since the embedding of
$V_{1}$ in $V_{0}$ is compact, we deduce that $S_{\varepsilon}(t)$ maps a
bounded set in $V_{0}$ into a compact set in $V_{0}$. In addition, the
operators $S_{\varepsilon}(t)$ are uniformly compact for $t\geq t_{1}\left(
B,R_{0}\right)  $. That is,%
\[%
{\textstyle\bigcup_{t\geq t_{1}}}
S_{\varepsilon}(t,0,B_{R_{0}})
\]
is relatively compact in $V_{0}$.

Due to a the standard procedure (cf., for example, \cite[Theorem I.1.1]{32}
for details), one can prove that there is a global attractor a compact
attractor $\mathfrak{A}_{\varepsilon}$ for the operators $S_{\varepsilon}(t)$
for $\varepsilon\geq0$,

Note that the global attractor $\mathfrak{A}_{\varepsilon}$ must be contained
in the absorbing balls $V_{0}$ and $V_{1}$%
\begin{equation}
\mathfrak{A}_{\varepsilon}=%
{\textstyle\bigcap_{t_{1}\geq0}}
\overline{%
{\textstyle\bigcup_{t\geq t_{1}}}
B_{\delta_{1}}\left(  t\right)  }\subset B_{\delta_{0}}\cap B_{\delta_{1}}.
\label{23}%
\end{equation}
Notice that all the above bounds are independent of $\varepsilon$.

\section{Estimates of Dimensions of the Global Attractor}

Our aim in this section is to study the finite dimensionality of the global
attractor. In the first part we will prove the differentiability property of
$S_{\varepsilon}\left(  t\right)  $ and in the second part we will provide
estimates of the fractal and Hausdorff dimensions of their global attractors
$\mathfrak{A}_{\varepsilon}$.

Using the trace formula \cite[Chapters V and VI]{32}, we estimate the
Hausdorff \ and the fractal dimensions of the global attractor $\mathfrak{A}%
_{\varepsilon}$ in $V$.

For a solution $u_{\varepsilon}\left(  t\right)  =S_{\varepsilon}\left(
t\right)  u_{\varepsilon0}$, $t\geq0$, lying on the attractor $u_{\varepsilon
0}\in\mathfrak{A}_{\varepsilon}$, we see from $($\ref{13}$)$ that the
linearized flow around $u_{\varepsilon}$ is given by the equation
\begin{equation}%
\begin{array}
[c]{r}%
U_{\varepsilon}^{\prime}+\varepsilon A^{2}U_{\varepsilon}+\nu AU_{\varepsilon
}+B\left(  u_{\varepsilon},U_{\varepsilon}\right)  +B\left(  U_{\varepsilon
},u_{\varepsilon}\right)  =0,\text{ in }V^{\prime}\\
U_{\varepsilon}\left(  0\right)  =\xi,\text{ in }V\text{.}%
\end{array}
\label{24}%
\end{equation}
We show the differentiability of the semigroup $S_{\varepsilon}$ with respect
to the initial data\ in the space $V$.

\begin{theorem}
\label{Theorem6}For any $t>0$, the function $u_{\varepsilon0}\rightarrow
u_{\varepsilon}\left(  t\right)  =S_{\varepsilon}\left(  t\right)
u_{\varepsilon0}$ is Fr\'{e}chet differentiable on the attractor
$\mathfrak{A}_{\varepsilon}$. Its differential is the linear operator
\[
D\left(  S_{\varepsilon}\left(  t\right)  u_{\varepsilon0}\right)  =L\left(
t,u_{\varepsilon0}\right)  :\xi\in V\rightarrow U_{\varepsilon}\left(
t\right)  \in V\text{,\ }t\in\left[  0,T\right]  \text{,}%
\]
where $U_{\varepsilon}\left(  t\right)  $ is the solution of $($\ref{24}$)$.
\end{theorem}

\begin{proof}
Let%
\[
\eta(t)=v_{\varepsilon}\left(  t\right)  -u_{\varepsilon}\left(  t\right)
-U_{\varepsilon}\left(  t\right)  \text{, }U_{\varepsilon}\left(  0\right)
=\xi=v_{\varepsilon0}-u_{\varepsilon0}.
\]
Clearly, $\eta$ satisfies%
\[
\eta_{t}+\varepsilon A^{2}\eta+\nu A\eta+B(\eta,v_{\varepsilon}%
)+B(v_{\varepsilon},\eta)-B(w_{\varepsilon},w_{\varepsilon})=0,\text{ }%
\eta(0)=0
\]
where $w_{\varepsilon}=v_{\varepsilon}-u_{\varepsilon}$. Taking the inner
product of the last equation with $\eta$ and using the identity
$B(v_{\varepsilon},\eta,\eta)=0$ we obtain%
\begin{equation}
\frac{d\left\Vert \eta\right\Vert ^{2}}{dt}+2\varepsilon\left\Vert
A\eta\right\Vert ^{2}+2\nu\left\Vert \eta\right\Vert _{1}^{2}=2b(\eta
,v_{\varepsilon},\eta)-2b(w_{\varepsilon},w_{\varepsilon},\eta). \label{25}%
\end{equation}
By $($\ref{4}$)$ the first term in the right-hand side of $($\ref{25}$)$ has
the estimate%
\begin{align*}
\left\vert 2b(\eta,v_{\varepsilon},\eta)\right\vert  &  \leq2c_{1}\left\Vert
\eta\right\Vert ^{\frac{1}{2}}\left\Vert \eta\right\Vert _{1}^{\frac{3}{2}%
}\left\Vert v_{\varepsilon}\right\Vert _{1}\\
&  \leq2c_{1}R_{1}\left\Vert \eta\right\Vert ^{\frac{1}{2}}\left\Vert
\eta\right\Vert _{1}^{\frac{3}{2}}\\
&  \leq\frac{c_{1}^{4}R_{1}^{4}}{\nu^{3}}\left\Vert \eta\right\Vert ^{2}%
+\frac{3\nu}{4}\left\Vert \eta\right\Vert _{1}^{2}.
\end{align*}
Employing the inequalities $($\ref{4}$)$ we estimate the second term in the
right hand side of $($\ref{25}$)$ as follows
\begin{align*}
2b(w_{\varepsilon},w_{\varepsilon},\eta)  &  \leq2c_{1}\left\Vert
\eta\right\Vert _{1}\left\Vert w_{\varepsilon}\right\Vert _{1}^{2}\\
&  \leq\frac{2c_{1}^{2}}{\nu}\left\Vert w_{\varepsilon}\right\Vert _{1}%
^{4}+\frac{\nu}{2}\left\Vert \eta\right\Vert _{1}^{2}.
\end{align*}
Hence, we obtain from $($\ref{25}$)$%
\[
\frac{d\left\Vert \eta\right\Vert ^{2}}{dt}+2\varepsilon\left\Vert
A\eta\right\Vert ^{2}+\frac{3\nu}{4}\left\Vert \eta\right\Vert _{1}^{2}%
\leq\frac{c_{1}^{4}R_{1}^{4}}{\nu^{3}}\left\Vert \eta\right\Vert ^{2}%
+\frac{2c_{1}^{2}}{\nu}\left\Vert w_{\varepsilon}\right\Vert _{1}^{4}%
\]
we drop the positive terms $2\varepsilon\left\Vert A\eta\right\Vert ^{2}$ and
$\frac{3\nu}{4}\left\Vert \eta\right\Vert _{1}^{2}$ we get%
\begin{equation}
\frac{d\left\Vert \eta\right\Vert ^{2}}{dt}\leq\frac{c_{1}^{4}R_{1}^{4}}%
{\nu^{3}}\left\Vert \eta\right\Vert ^{2}+\frac{2c_{1}^{2}}{\nu}\left\Vert
w_{\varepsilon}\right\Vert _{1}^{4}. \label{26}%
\end{equation}
From the classical Gronwall Lemma (see \cite{33}), $($\ref{26}$)$ gives
\[
\left\Vert \eta\right\Vert ^{2}\leq\frac{2c_{1}^{2}}{\nu}\int_{0}%
^{t}\left\Vert w_{\varepsilon}\right\Vert _{1}^{4}\exp(\int_{s}^{t}\frac
{c_{1}^{4}R_{1}^{4}}{\nu^{3}}d\tau)ds.
\]
Thus
\begin{equation}
\left\Vert \eta\right\Vert ^{2}\leq C_{0}\int_{0}^{t}\left\Vert w_{\varepsilon
}\right\Vert _{1}^{4}ds,\text{ }C_{0}=\frac{2c_{1}^{2}}{\nu}\exp(\frac
{Tc_{1}^{4}R_{1}^{4}}{\nu^{3}}). \label{27}%
\end{equation}
The difference%
\[
w_{\varepsilon}\left(  t\right)  =v_{\varepsilon}\left(  t\right)
-u_{\varepsilon}\left(  t\right)  =S_{\varepsilon}\left(  t\right)
v_{\varepsilon0}-S_{\varepsilon}(t)u_{\varepsilon0}%
\]
satisfies the equation%
\begin{equation}
\frac{dw_{\varepsilon}}{dt}+\varepsilon A^{2}w_{\varepsilon}+\nu
Aw_{\varepsilon}+B(w_{\varepsilon},v_{\varepsilon})+B(v_{\varepsilon
},w_{\varepsilon})-B(w_{\varepsilon},w_{\varepsilon})=0 \label{27a}%
\end{equation}
and
\[
w_{\varepsilon}(0)=v_{\varepsilon0}-u_{\varepsilon0}=w_{\varepsilon0}.
\]
Taking the inner product of the last equation with $w_{\varepsilon}$,we
obtain
\begin{equation}
\frac{d}{dt}\left\Vert w_{\varepsilon}\right\Vert ^{2}+2\varepsilon\left\Vert
Aw_{\varepsilon}\right\Vert ^{2}+2\nu\left\Vert w_{\varepsilon}\right\Vert
_{1}^{2}=2b(w_{\varepsilon},w_{\varepsilon},v_{\varepsilon}). \label{28}%
\end{equation}
By using inequalities $($\ref{4}$)$ and Young's inequality we obtain
\begin{align*}
\left\vert 2b(w_{\varepsilon},v_{\varepsilon},w_{\varepsilon})\right\vert  &
\leq2c_{1}\left\Vert v_{\varepsilon}\right\Vert _{1}\left\Vert w_{\varepsilon
}\right\Vert _{1}^{\frac{3}{2}}\left\Vert w_{\varepsilon}\right\Vert
^{\frac{1}{2}}\\
&  \leq\frac{c_{1}^{4}R_{1}^{4}}{\nu^{3}}\left\Vert w_{\varepsilon}\right\Vert
^{2}+\frac{3\nu}{4}\left\Vert w_{\varepsilon}\right\Vert _{1}^{2}.
\end{align*}
Substituting the above result into $($\ref{28}$)$, we obtain%
\begin{equation}
\frac{d}{dt}\left\Vert w_{\varepsilon}\right\Vert ^{2}+2\varepsilon\left\Vert
Aw_{\varepsilon}\right\Vert ^{2}+\frac{5\nu}{4}\left\Vert w_{\varepsilon
}\right\Vert _{1}^{2}\leq\frac{c_{1}^{4}R_{1}^{4}}{\nu^{3}}\left\Vert
w_{\varepsilon}\right\Vert ^{2}\text{.} \label{29}%
\end{equation}
We drop the positive terms $2\varepsilon\left\Vert Aw_{\varepsilon}\right\Vert
^{2}$ and $\frac{5\nu}{4}\left\Vert w_{\varepsilon}\right\Vert _{1}^{2}$ to
obtain the following differential inequality%
\begin{equation}
\frac{d}{dt}\left\Vert w_{\varepsilon}\right\Vert ^{2}\leq\frac{c_{1}^{4}%
R_{1}^{4}}{\nu^{3}}\left\Vert w_{\varepsilon}\right\Vert ^{2}. \label{30}%
\end{equation}
Using the classical Gronwall Lemma we deduce from $($\ref{30}$)$ that%
\begin{equation}
\left\Vert w_{\varepsilon}\right\Vert ^{2}\leq\left\Vert w_{\varepsilon
}\left(  0\right)  \right\Vert ^{2}\exp(\frac{Tc_{1}^{4}R_{1}^{4}}{\nu^{3}%
})\text{.} \label{31}%
\end{equation}
From $($\ref{31}$)$ we deduce that%
\begin{equation}
\int_{0}^{t}\left\Vert u_{\varepsilon}\left(  t\right)  -v_{\varepsilon
}\left(  t\right)  \right\Vert _{1}^{2}dt\leq C_{1}\left\Vert u_{\varepsilon
0}-v_{\varepsilon0}\right\Vert ^{2}\text{; }C_{1}=\frac{4}{5\nu}T\exp
(\frac{Tc_{1}^{4}R_{1}^{4}}{\nu^{3}}), \label{32}%
\end{equation}
with $($\ref{27}$)$ we conclude that
\[
\left\Vert \eta\right\Vert ^{2}\leq C_{0}C_{1}^{2}\left\Vert u_{\varepsilon
0}-v_{\varepsilon0}\right\Vert ^{4},
\]
then we deduce from $($\ref{27}$)$ and $($\ref{32}$)$ that%
\begin{equation}
\left\Vert \eta\right\Vert ^{2}\leq C_{2}\left\Vert w_{\varepsilon}\left(
0\right)  \right\Vert ^{4}\text{, where\ }C_{2}=C_{0}C_{1}^{2} \label{33}%
\end{equation}
this shows that%
\[
\frac{\left\Vert v_{\varepsilon}(t)-u_{\varepsilon}(t)-U_{\varepsilon
}(t)\right\Vert ^{2}}{\left\Vert v_{\varepsilon0}-u_{\varepsilon0}\right\Vert
^{2}}\leq C_{2}\left\Vert v_{\varepsilon0}-u_{\varepsilon0}\right\Vert
^{2}\rightarrow0\text{ as }\left\Vert v_{\varepsilon0}-u_{\varepsilon
0}\right\Vert _{1}\rightarrow0\text{, on }\mathfrak{A}_{\varepsilon}\text{.}%
\]
The differentiability of $S_{\varepsilon}\left(  t\right)  $ is proved.
\end{proof}

From Theorem \ref{Theorem6}\ the function $S_{\varepsilon}\left(  t\right)  $
is Fr\'{e}chet differentiable on $\mathfrak{A}_{\varepsilon}$ for $t>0$.

For $\xi\in V_{0}$, there exists a unique solution $U_{\varepsilon}$ of
$($\ref{24}$)$ satisfies%
\[
U_{\varepsilon}\in C\left(  \left[  0,T\right]  ;V_{0}\right)  \cap
L^{2}\left(  0,T;V_{2}\right)  \text{ \ }\forall T>0.
\]

With the differentiability ensured in Theorem \ref{Theorem5} we can then
define a linear map $L\left(  t;u_{\varepsilon0}\right)  :\xi\in
V_{0}\rightarrow U_{\varepsilon}\left(  t\right)  \in V_{0}$ where
$U_{\varepsilon}$ is the solution of $($\ref{24}$)$.

We can apply the trace formula (see \cite{8} and \cite[Section V. 3]{32}) to
find a bound on the dimension of the global attractor $\mathfrak{A}%
_{\varepsilon}$. We consider the trace $TrF^{\prime}\left(  u_{\varepsilon
}\right)  $ of the linear operator $F^{\prime}\left(  u_{\varepsilon}\right)
$\ and for\ $m\in%
\mathbb{N}
$, the number%
\[
q_{m}=\lim\sup_{t\rightarrow\infty}\sup_{u_{\varepsilon0}\in A}\sup
_{\substack{\xi_{1}\in V_{0}\\\left\vert \xi_{1}\right\vert \leq
1\\i=1,...,m}}\frac{1}{t}\int_{0}^{t}TrF^{\prime}(S_{\varepsilon}\left(
\tau\right)  u_{\varepsilon0})\circ Q_{m}\left(  \tau\right)  d\tau
\]
where $Q_{m}\left(  \tau\right)  =Q_{m}\left(  \tau,u_{\varepsilon0};\xi
_{1},...,\xi_{m}\right)  $ is the orthogonal projector in $V_{0}$ onto the
space spanned by $U_{\varepsilon}^{1}\left(  \tau\right)  ,...,U_{\varepsilon
}^{m}\left(  \tau\right)  $. where $U_{\varepsilon}^{j}\left(  \tau\right)  $
$=L\left(  \tau,u_{\varepsilon0}\right)  .\xi_{j}$, $j=1,...,m$, $t\geq0,$
$\ $are $m$ solutions of $($\ref{24}$)$, corresponding to $\xi=\xi_{1}%
,...,\xi_{m}\in V_{1}$. Let $\varphi_{j}\left(  \tau\right)  $, $j=1,...,m$,
$\tau\geq0$, be an orthonormal basis of for $\tilde{Q}_{m}\left(  \tau\right)
V_{0}=$span $\left\{  U_{\varepsilon}^{1}\left(  \tau\right)
,...,U_{\varepsilon}^{m}\left(  \tau\right)  \right\}  $, $\varphi_{j}\left(
t\right)  \in V_{1}$ for $j=1,...,m$, since $U_{\varepsilon}^{1}\left(
\tau\right)  ,...,U_{\varepsilon}^{m}\left(  \tau\right)  \in V_{1}$, $\tau\in%
\mathbb{R}
^{+}$.

From the general result in \cite[Section V.3.41]{32}, we have that if
$q_{m}<0$ for some $m\in N$ then the global attractor has finite Hausdorff and
fractal dimensions estimated respectively as%
\begin{align}
\dim_{H}\left(  \mathfrak{A}_{\varepsilon}\right)   &  \leq m,\label{34}\\
\dim_{F}\left(  \mathfrak{A}_{\varepsilon}\right)   &  \leq m(1+\max_{1\leq
j\leq m-1}\frac{\left(  q_{j}\right)  _{+}}{\left\Vert q_{m}\right\Vert }).
\label{35}%
\end{align}
Then we have%
\begin{align*}
TrF^{\prime}\left(  S_{\varepsilon}\left(  \tau\right)  u_{\varepsilon
0}\right)  \circ Q_{m}\left(  \tau\right)   &  =%
{\textstyle\sum_{j=1}^{\infty}}
\left(  TrF^{\prime}\left(  u_{\varepsilon}\left(  \tau\right)  \right)  \circ
Q_{m}\left(  \tau\right)  \varphi_{j}\left(  \tau\right)  ,\varphi_{j}\left(
\tau\right)  \right) \\
&  =%
{\textstyle\sum\limits_{j=1}^{m}}
\left(  F^{\prime}\left(  u_{\varepsilon}\left(  \tau\right)  \right)
\varphi_{j}\left(  \tau\right)  ,\varphi_{j}\left(  \tau\right)  \right)  .
\end{align*}
Recall that $(.,.)$ denoting the scalar product in $V_{0}$, we write using
$($\ref{2}$)$ and $($\ref{3}$)$%

\begin{align*}
Tr(F^{\prime}\left(  u_{\varepsilon}\left(  \tau\right)  \right)  \varphi
_{j}\left(  \tau\right)  ,\varphi_{j}\left(  \tau\right)  )  &  =%
{\textstyle\sum\limits_{j=1}^{m}}
\left(  -\varepsilon A^{2}\varphi_{j}-\nu A\varphi_{j}-B(\varphi
_{j},u_{\varepsilon})-B(u_{\varepsilon},\varphi_{j}),\varphi_{j}\right) \\
&  =%
{\textstyle\sum\limits_{j=1}^{m}}
(-\varepsilon\left\Vert A\varphi_{j}\right\Vert ^{2}-\nu\Vert A^{\frac{1}{2}%
}\varphi_{j}\Vert^{2}-b\left(  u_{\varepsilon},\varphi_{j},\varphi_{j}\right)
-b\left(  \varphi_{j},u_{\varepsilon},\varphi_{j}\right)  )
\end{align*}
thus%
\begin{equation}
Tr\left(  F^{\prime}\left(  u_{\varepsilon}\left(  \tau\right)  \right)
\varphi_{j}\left(  \tau\right)  ,\varphi_{j}\left(  \tau\right)  \right)  =%
{\textstyle\sum_{j=1}^{m}}
(-\varepsilon\left\Vert \varphi_{j}\right\Vert _{2}^{2}-\nu\left\Vert
\varphi_{j}\right\Vert _{1}^{2}-b\left(  \varphi_{j},u_{\varepsilon}%
,\varphi_{j}\right)  ). \label{36}%
\end{equation}
We estimate the nonlinear term as follows%
\begin{equation}
\mid%
{\textstyle\sum\limits_{j=1}^{m}}
b\left(  \varphi_{j},u,\varphi_{j}\right)  \mid=\mid%
{\textstyle\sum\limits_{j=1}^{m}}
\int_{\Omega}%
{\textstyle\sum\limits_{i,k=1}^{3}}
\varphi_{ji}\frac{\partial u_{_{k}}}{\partial x_{i}}\left(  x\right)
\varphi_{jk}dx\mid\label{38}%
\end{equation}
whence for almost every $x\in\Omega$ we have%
\[
\mid%
{\textstyle\sum\limits_{j=1}^{m}}
{\textstyle\sum\limits_{i,k=1}^{3}}
\varphi_{ji}\frac{\partial u_{_{k}}}{\partial x_{i}}\left(  x\right)
\varphi_{jk}dx\mid\leq\left\Vert u\right\Vert _{1}\left\Vert \rho\right\Vert
\]
where%
\[
\left\Vert u\left(  x\right)  \right\Vert _{1}=(%
{\textstyle\sum\limits_{i,k=1}^{3}}
\left\Vert D_{i}u_{_{k}}\left(  x\right)  \right\Vert ^{2})^{\frac{1}{2}}%
\]
and
\begin{equation}
\rho\left(  x\right)  =%
{\textstyle\sum\limits_{j=1}^{m}}
{\textstyle\sum\limits_{i=1}^{3}}
\left(  \varphi_{ji}\left(  x\right)  \right)  ^{2}. \label{39}%
\end{equation}
Therefore%
\begin{equation}
\mid%
{\textstyle\sum\limits_{j=1}^{m}}
b\left(  \varphi_{j},u_{\varepsilon},\varphi_{j}\right)  \mid\leq\int_{\Omega
}\rho\left(  x\right)  \left\Vert u_{\varepsilon}\left(  x\right)  \right\Vert
_{1}dx. \label{37}%
\end{equation}

Now we recall the generalized form of the Lieb--Thirring inequality in
dimension three and $m=l$ as developed in \cite[Theorem A4.1]{32}

\begin{theorem}
(The Lieb--Thirring inequality). Let $\varphi_{j}$, $1\leq j\leq N$ be a
finite family of $V_{l}$ wich is orthonormal in $L^{2}(\Omega)$ and set, for
every $x\in\Omega,$%
\[
\rho\left(  x\right)  =%
{\textstyle\sum\limits_{j=1}^{N}}
\left\Vert \left(  \varphi_{j}\left(  x\right)  \right)  \right\Vert ^{2}%
\]
Then there exists a constant $\kappa$, independent of the family $\varphi_{j}$
and of $N$ such that%
\begin{equation}
(\int_{\Omega}\rho\left(  x\right)  ^{q/q-1}dx)^{2l(q-1)/3}\leq\kappa%
{\textstyle\sum\limits_{j=1}^{N}}
\int_{\Omega}a(\varphi_{j},\varphi_{j}). \label{41}%
\end{equation}
for all $q\in max\{(1,3/2l),(1+3/2l)\}$ and where $\kappa$ depends on $l$,
$p$, and $q$, and on the shape (but not the size) of $\Omega$.
\end{theorem}

The quadratic form we will use is
\begin{equation}
a(v,u)=(A^{l}v,u)=(A^{l/2}v,A^{l/2}u) \label{40}%
\end{equation}
so that the order of our quadratic form is $l$.

Kolmogorov's mean rate of dissipation of energy in turbulent flow (see e.g.
\cite[VI.(3.20)]{11,16,32}) is defined as%
\begin{equation}
\epsilon=\lambda_{1}^{\frac{3}{2}}\nu\lim\sup_{t\rightarrow\infty}%
\sup_{u_{\varepsilon0}\in\mathfrak{A}_{\varepsilon}}\frac{1}{t}\int_{0}%
^{t}\left\Vert u_{\varepsilon}\left(  \tau\right)  \right\Vert _{1}^{2}%
d\tau\label{42}%
\end{equation}
the maximal mean rate of dissipation of energy on the attractor, which is
finite thanks to $($\ref{20}$)$.

Using $($\ref{20}$)$ we can estimate the energy dissipation flux $\epsilon$ by%
\begin{equation}
\epsilon\leq\frac{\lambda_{1}^{\frac{1}{2}}\left\Vert f\right\Vert ^{2}}{\nu}.
\label{43}%
\end{equation}
In order to make the dimension estimate more explicit, we can estimate the
energy dissipation flux $\epsilon$ in terms of $G$ by%
\begin{equation}
\epsilon\leq\lambda_{1}^{2}\nu^{3}G^{2}. \label{44}%
\end{equation}

\section{Numbers of degrees of freedom in turbulent flows}

In this Section, we estimate the effects of hyperviscosity on the turbulent
flow. An argument from the classical theory of turbulence (see, L. Landau and
Lifshitz \cite{22}) suggests that there are finitely many degrees of freedom
in turbulent flows. Heuristic physical arguments are used to justify this
assertion and to provide an estimate for this number of degrees of freedom by
dividing a typical length scale of the flow, $l_{0}=\lambda_{1}^{-\frac{1}{2}%
}$, by the Kolmogorov dissipation length scale $l_{\epsilon}$ i.e.
$l_{\epsilon}=\frac{\nu^{3}}{\epsilon}$ where $\epsilon$\ is Kolmogorov's mean
rate of dissipation of energy in turbulent flow and taking the third power in 3D.

We will express our primary attractor results in terms of the Kolmogorov
length-scale $l_{\epsilon}$ and the Landau-Lifschitz estimates \cite{22} of
the number of degrees of freedom in turbulent flow \cite{11, 32} and we can
easily observe such compatibility that exists between these estimates and the
number of degrees of freedom in turbulence (see also \cite{22}). Such
estimates will give us useful information about the capability of $($%
\ref{13}$)$ to approximate Navier-Stokes equations dynamics. We will show that
the corresponding number of degrees of freedom is proportional to the
dimension of the global attractor.

By Holder's inequality the right hand side of $($\ref{37}$)$ can be estimated
as follow%
\begin{equation}
\int_{\Omega}\left\Vert u_{\varepsilon}\left(  x\right)  \right\Vert _{1}%
\rho\left(  x\right)  dx\leq\left\Vert \rho\left(  x\right)  \right\Vert
_{L^{\frac{7}{3}}\left(  \Omega\right)  }\parallel A^{\frac{1}{2}%
}u_{\varepsilon}\left(  x\right)  \parallel_{L^{\frac{7}{4}}\left(
\Omega\right)  } \label{58}%
\end{equation}
Applying Young's inequality with
\begin{equation}
p=\frac{7}{3},\text{ }q=\frac{7}{4},\text{ }\sigma=\frac{7\varepsilon}%
{6\kappa} \label{59}%
\end{equation}
we obtain%
\begin{equation}
\int_{\Omega}\left\Vert u_{\varepsilon}\left(  x\right)  \right\Vert _{1}%
\rho\left(  x\right)  dx\leq\frac{\varepsilon}{2\kappa}\left\Vert \rho\left(
x\right)  \right\Vert _{L^{\frac{7}{3}}\left(  \Omega\right)  }^{\frac{7}{3}%
}+c_{5}\parallel A^{\frac{1}{2}}u_{\varepsilon}\left(  x\right)
\parallel_{L^{\frac{7}{4}}\left(  \Omega\right)  }^{^{\frac{7}{4}}},\text{
}c_{5}=\frac{4}{7}(\frac{7\varepsilon}{6\kappa})^{-\frac{3}{4}}\text{\ \ \ \ }%
. \label{45}%
\end{equation}
Using $($\ref{45}$)$ we can majorize $TrF^{\prime}\left(  u_{\varepsilon
}\left(  \tau\right)  \right)  \circ\tilde{Q}_{m}\left(  \tau\right)  $ as
follows
\begin{align}
TrF^{\prime}\left(  u_{\varepsilon}\left(  \tau\right)  \right)  \circ
\tilde{Q}_{m}\left(  \tau\right)   &  \leq-\nu%
{\textstyle\sum\limits_{j=1}^{m}}
\left\Vert \varphi_{j}\left(  x\right)  \right\Vert _{1}^{2}-\varepsilon%
{\textstyle\sum\limits_{j=1}^{m}}
\left\Vert \varphi_{j}\left(  \tau\right)  \right\Vert _{2}^{2}+\frac
{\varepsilon}{2\kappa}\left\Vert \rho\left(  x\right)  \right\Vert
_{L^{\frac{7}{3}}\left(  \Omega\right)  }^{\frac{7}{3}}\label{46}\\
+c_{5}  &  \parallel A^{\frac{1}{2}}u_{\varepsilon}\left(  x\right)
\parallel_{L^{\frac{7}{4}}\left(  \Omega\right)  }^{^{\frac{7}{4}}}\text{.}%
\end{align}
Applying the Lieb--Thirring inequality (%
\ref{41}
) we obtain%

\[
TrF^{\prime}\left(  u_{\varepsilon}\left(  \tau\right)  \right)  \circ
\tilde{Q}_{m}\left(  \tau\right)  \leq-\frac{\nu}{2}%
{\textstyle\sum\limits_{j=1}^{m}}
\left\Vert \varphi_{j}\left(  x\right)  \right\Vert _{1}^{2}-\frac
{\varepsilon}{2}%
{\textstyle\sum\limits_{j=1}^{m}}
\left\Vert \varphi_{j}\left(  \tau\right)  \right\Vert _{2}^{2}+c_{5}\parallel
A^{\frac{1}{2}}u_{\varepsilon}\left(  x\right)  \parallel_{L^{\frac{7}{4}%
}\left(  \Omega\right)  }^{^{\frac{7}{4}}}.
\]
The Sobolev embedding $V_{2}\subset V_{1}$ the Sobolev inequalities on $%
\Omega
$ in terms of
\[
\parallel\varphi_{j}\left(  x\right)  \parallel_{1}\leq c_{6}\parallel
\varphi_{j}\left(  x\right)  \parallel_{2}%
\]
we get%
\begin{equation}
TrF^{\prime}\left(  u_{\varepsilon}\left(  \tau\right)  \right)  \circ
\tilde{Q}_{m}\left(  \tau\right)  \leq-(\frac{\nu}{2}+\frac{\varepsilon
}{2c_{6}})%
{\textstyle\sum\limits_{j=1}^{m}}
\left\Vert \varphi_{j}\left(  x\right)  \right\Vert _{1}^{2}+c_{5}\parallel
A^{\frac{1}{2}}u_{\varepsilon}\left(  x\right)  \parallel_{L^{\frac{7}{4}%
}\left(  \Omega\right)  }^{^{\frac{7}{4}}}. \label{60}%
\end{equation}
then,%
\begin{equation}
TrF^{\prime}\left(  u_{\varepsilon}\left(  \tau\right)  \right)  \circ
\tilde{Q}_{m}\left(  \tau\right)  \leq-c_{7}%
{\textstyle\sum\limits_{j=1}^{m}}
\left\Vert \varphi_{j}\left(  x\right)  \right\Vert _{1}^{2}+c_{5}\parallel
A^{\frac{1}{2}}u_{\varepsilon}\left(  x\right)  \parallel_{L^{\frac{7}{4}%
}\left(  \Omega\right)  }^{^{\frac{7}{4}}}\text{, where }c_{7}=\frac{\nu}%
{2}+\frac{\varepsilon}{2c_{6}}. \label{61}%
\end{equation}

Note that in the 3D case we have $\lambda_{j}\geq c_{8}L^{-2}j^{\frac{2}{3}}$
for some positive universal constant (see, for example \cite[Lemma VI 2.1]%
{32}). Therefore,%
\begin{equation}%
{\textstyle\sum\limits_{j=1}^{m}}
\left\Vert \varphi_{j}\left(  x\right)  \right\Vert _{1}^{2}\geq\lambda
_{1}+...+\lambda_{m}\geq c_{9}\lambda_{1}m^{\frac{5}{3}}. \label{47}%
\end{equation}
For the term $\parallel A^{\frac{1}{2}}u_{\varepsilon}\left(  x\right)
\parallel_{L^{\frac{7}{4}}\left(  \Omega\right)  }^{^{\frac{7}{4}}}$, we have
by Holder's inequality that%
\begin{equation}
\parallel A^{\frac{1}{2}}u_{\varepsilon}\left(  x\right)  \parallel
_{L^{\frac{7}{4}}\left(  \Omega\right)  }^{^{\frac{7}{4}}}\leq c_{10}\parallel
A^{\frac{1}{2}}u_{\varepsilon}\left(  x\right)  \parallel^{^{\frac{7}{4}}%
}\text{ with }c_{10}=\left\vert \Omega\right\vert ^{\frac{1}{8}} \label{48}%
\end{equation}
Taking into account $($\ref{46}$)$ then yields
\begin{equation}
TrF^{\prime}\left(  u_{\varepsilon}\left(  \tau\right)  \right)  \circ
Q_{m}\left(  \tau\right)  d\tau\leq-c_{7}c_{9}\lambda_{1}m^{\frac{5}{3}}%
+c_{5}c_{10}\parallel A^{\frac{1}{2}}u_{\varepsilon}\left(  x\right)
\parallel^{\frac{7}{4}}. \label{49}%
\end{equation}
By H\"{o}lder's inequality we have%
\begin{equation}
\lim\sup_{t\rightarrow\infty}\sup_{u_{\varepsilon0}\in\mathfrak{A}%
_{\varepsilon}}\frac{1}{t}\int_{0}^{t}\parallel A^{\frac{1}{2}}u_{\varepsilon
}\left(  \tau,x\right)  \parallel^{\frac{7}{4}}d\tau\leq\lim\sup
_{t\rightarrow\infty}(\sup_{u_{\varepsilon0}\in\mathfrak{A}_{\varepsilon}%
}\frac{1}{t}\int_{0}^{t}\parallel A^{\frac{1}{2}}u_{\varepsilon}\left(
\tau,x\right)  \parallel^{2}d\tau)^{\frac{7}{8}} \label{50}%
\end{equation}
On the other hand, using $($\ref{42}$)$ we have%
\begin{equation}
\lim\sup_{t\rightarrow\infty}\sup_{u_{\varepsilon0}\in\mathfrak{A}%
_{\varepsilon}}\frac{1}{t}\int_{0}^{t}\parallel A^{\frac{1}{2}}u_{\varepsilon
}\left(  \tau,x\right)  \parallel^{\frac{7}{4}}d\tau\leq(\frac{\epsilon
}{\lambda_{1}^{\frac{3}{2}}\nu})^{\frac{7}{8}}. \label{51}%
\end{equation}
For $u_{\varepsilon0}\in\mathfrak{A}_{\varepsilon}$, we can estimate the
quantities $q_{m}\left(  t\right)  $, $q_{m}$\
\begin{equation}
q_{m}=\lim\sup_{t\rightarrow\infty}q_{m}\left(  t\right)  \leq-\kappa
_{1}m^{\frac{5}{3}}+\kappa_{2}, \label{52}%
\end{equation}
where%
\begin{equation}
\kappa_{1}=c_{7}c_{9}\lambda_{1}\text{ and\ }\kappa_{2}=c_{5}c_{10}%
(\frac{\epsilon}{\lambda_{1}^{\frac{3}{2}}\nu})^{\frac{7}{8}}. \label{62}%
\end{equation}
Therefore, if $m^{\prime}\in%
\mathbb{N}
$ is defined by
\begin{equation}
m^{\prime}-1<(\frac{2\kappa_{2}}{\kappa_{1}})^{\frac{3}{5}}=(\frac
{2c_{5}c_{10}}{c_{7}c_{9}\lambda_{1}^{\frac{37}{16}}\nu^{\frac{7}{8}}}%
)^{\frac{3}{5}}\epsilon^{\frac{21}{40}}<m^{\prime}, \label{53}%
\end{equation}
Setting $l_{\epsilon}=(\frac{\nu^{3}}{\epsilon})^{\frac{1}{4}}$ the
dissipation length scale, and $l_{0}=\lambda_{1}^{-\frac{1}{2}}$ the
macroscopical length by setting.\ Then we can rewrite $($\ref{53}$)$ in the
form%
\begin{equation}
m^{\prime}-1<c_{_{11}}(\frac{l_{0}}{l_{\epsilon}})^{\frac{21}{10}}<m^{\prime}
\label{54}%
\end{equation}
where
\begin{equation}
c_{_{11}}=(\frac{2c_{5}c_{10}}{c_{7}c_{9}\lambda_{1}^{\frac{37}{16}}\nu
^{\frac{7}{8}}})^{\frac{3}{5}}\text{ }(\nu^{^{\frac{63}{40}}})\lambda
_{1}^{\frac{21}{20}}. \label{55}%
\end{equation}
Thus, we have proved the following Proposition

\begin{proposition}
The Hausdorff and fractal dimensions of the global attractor $\mathfrak{A}%
_{\varepsilon}$\ of the regularized 3D Navier-Stokes $($\ref{13}$)$, $\dim
_{F}\left(  \mathfrak{A}_{\varepsilon}\right)  $ and $\dim_{H}\left(
\mathfrak{A}_{\varepsilon}\right)  $ respectively, satisfy%
\begin{equation}
\dim_{H}\left(  \mathfrak{A}_{\varepsilon}\right)  \leq\dim_{F}\left(
\mathfrak{A}_{\varepsilon}\right)  \leq c_{_{11}}(\frac{l_{0}}{l_{\epsilon}%
})^{\frac{21}{10}}. \label{56}%
\end{equation}

\end{proposition}

The exponent on $\frac{l_{0}}{l_{\epsilon}}$ is significantly less than the
Landau--Lifschitz predicted value of $3$, less than the results in \cite{9}
for the 3D Camassa--Holm equations, or simply NS-$\alpha$ model and less than
the Avrin exponent (for $\alpha=l=2$) \cite[Theorem 1]{1}.

This, in a sense, suggests that in the absence of boundary effects (e.g., in
the case of periodic boundary conditions) the modified 3D Navier-Stokes
represent, very well, the averaged equation of motion of turbulent flows.

Since the Grashoff number $G=\dfrac{\left\Vert f\right\Vert }{\nu^{2}%
\lambda_{1}^{\frac{3}{4}}}$ in 3D, (see e.g. \cite{1,11, 33}) is an upper
bound for $(\frac{l_{0}}{l_{\epsilon}})^{2}$, expressing the above estimates
in terms of $G$ is straightforward. The above Proposition becomes

\begin{proposition}
The Hausdorff and fractal dimensions of the global attractor $\mathfrak{A}%
_{\varepsilon}$\ of the regularized 3D Navier-Stokes $($\ref{13}$)$, $\dim
_{F}\left(  \mathfrak{A}_{\varepsilon}\right)  $ and $\dim_{H}\left(
\mathfrak{A}_{\varepsilon}\right)  $ respectively, satisfy%
\begin{equation}
\dim_{H}\left(  \mathfrak{A}_{\varepsilon}\right)  \leq\dim_{F}\left(
\mathfrak{A}_{\varepsilon}\right)  \leq c_{_{11}}G^{\frac{21}{20}}\text{.}
\label{57}%
\end{equation}

\end{proposition}

Thus we recover the improvement on the cubic power, i.e. get a bound
proportional to $G^{\frac{p}{2}}$ for $p<3$, in $($\ref{57}$)$ $p=\frac
{21}{10}$. This improvement suggesting to very good agreement with the
conventional theory of turbulence.

For $\alpha=l=2$, motivated by the Chapman--Enskog expansion, we recover
$($\ref{57}$)$. This result can be seen as an improved version of the results
announced by Joel\ Avrin \cite[Theorem 2]{1}.

This upper bound is much smaller than what one would expect for
three-dimensional models, i.e. $(\frac{l_{0}}{l_{\epsilon}})^{3}$. This
improves significantly on previous bounds have demonstrated that
hyperviscosity can have profound effects on the number of degree freedom. The
modifying effects are well understood, which makes the use of hyperviscosity
an efficient tool for numerical studies and suggests that the regularized 3D
Navier-Stokes has a great potential to become a good sub-gridscale large-eddy
simulation model of turbulence. The results obtained agree very well with
those provided in numerical studies of turbulence(see, Refs.,\cite{1}%
,\cite{9}, \cite{13}, \cite{15}, \cite{21}).

The present results explain some fundamental differences between the theory
use instead a hyper-viscous term to approximate Navier-Stokes equations and
which hyperviscous terms are added spectrally to the standard incompressible
Navier--Stokes equations \cite{1}. It would be interesting to obtain estimates
for $($\ref{1}$)$ in this context in 3D and to see how the estimates depend on
$l$ for $l\geq\frac{3}{2}$.

\end{document}